\documentclass[10pt]{amsart}
\usepackage[paperwidth=160mm,paperheight=220mm,total={12cm,19.3cm},top=15mm,left=23mm,includeheadfoot]{geometry}

\usepackage[utf8]{inputenc}
\usepackage{amsmath}
\usepackage{amsfonts}
\usepackage{amssymb,amsthm,mathrsfs}
\usepackage[firstinits=true,sorting=none,style=numeric-comp,backend=bibtex]{biblatex}

\usepackage{tikz}

\usepackage{graphicx}
\usepackage{subcaption}
\graphicspath{{./figuresMod2/}}

\usepackage[colorlinks=true,linkcolor=blue,citecolor=blue]{hyperref}

\bibliography{model2}

\newtheorem{teo}{Theorem}[section]

\newtheorem{lem}[teo]{Lemma}

\theoremstyle{remark}
\newtheorem{ejm}{Example}[section]
\newtheorem{rem}{Remark}[section]

\newcommand{\der}[1]{\frac{\mathrm{d}#1}{\mathrm{d}t}}
\newcommand{\ds}{\der{S}}
\newcommand{\di}{\der{I}}
\newcommand{\dr}{\der{R}}

\title[Global stability for SIRS epidemic models with general \ldots]{Global stability for SIRS epidemic models with general incidence rate and transfer from infectious to susceptible}
\author[Ángel G. Cervantes-Pérez, Eric J. Ávila-Vales]{}

\begin{document}
    
\maketitle

\begin{center}
    {\small \textsc{Ángel G. Cervantes-Pérez$^1$}, \textsc{Eric J. Ávila-Vales$^2$}
    }
\end{center}
\bigskip
\begin{center} {\small \sl $^{1,2}$ Facultad de Matem\'aticas, Universidad Aut\'onoma de Yucat\'an, \\ Anillo Perif\'erico Norte, Tablaje 13615, C.P. 97119, M\'erida, Yucat\'an, Mexico}
\end{center}

\bigskip

\begin{abstract}
    We study a class of SIRS epidemic dynamical models with a general non-linear incidence rate and transfer from infectious to susceptible. The incidence rate includes a wide range of monotonic, concave incidence rates and some non-monotonic or concave cases. We apply LaSalle’s invariance principle and Lyapunov's direct method to prove that the disease-free equilibrium is globally asymptotically stable if the basic reproduction number $R_0\le1$, and the endemic equilibrium is globally asymptotically stable if $R_0>1$, under some conditions imposed on the incidence function $f(S,I)$.
\end{abstract}

\section{Introduction}

In the theory of epidemic mathematical models, the SIRS (susceptible-infected-removed) model is one of the most important ones. In recent years, the global stability of equilibria for SIRS autonomous epidemic models has received a lot of attention. Kermack and McKendrick \cite{kermack1927contribution} studied one of the simplest epidemic models of this type. From then on, a number of researchers have made studies on them \cite{lahrouz2011global,korobeinikov2006lyapunov,buonomo2010lyapunov,muroya2011global,vargas2011global,xu2009stability}. The incidence rate of a disease measures how fast the disease is spreading, and it plays an important role in the research of mathematical epidemiology. In many previous epidemic models, the bilinear incidence rate $\beta SI$ was frequently used \cite{hethcote2000mathematics,kermack1927contribution,mena1992dynamic,guo2006global}. However, there are some advantages for adopting more general forms of incidence rates. For instance, Capasso and his co-workers observed in the seventies \cite{capasso1978global,capasso1977modelli,capasso1978generalization} that the incidence rate may increase more slowly as $I$ increases, so they proposed a saturated incidence rate.

\citeauthor{ruan2003dynamical} \cite{ruan2003dynamical} studied an epidemic model with a non-linear incidence rate of the form $\beta I^2S/(1+\rho I^2)$. Liu et al. \cite{liu1986influence} initially introduced a general incidence rate $\beta I^pS/(1+\rho I^q)$. \citeauthor{lahrouz2012complete} \cite{lahrouz2012complete} proposed a more general incidence rate $\beta SI/\psi(I)$, which includes the cases that model the ``psychological effect'': when the number of infectives is large, the force of infection may decrease as the number of infectives increases. This happens when $I/\psi(I)$ is increasing for small $I$ but decreases when the value of $I$ gets larger.

As models with more general incidence functions are considered, the dynamics of the system become more complicated. Models with  incidence functions of the form $g(I)h(S)$ have been studied, such as \cite{korobeinikov2005non,tang2017new}. In the most general case, the transmission of the disease may be given by a non-factorable function of $S$ and $I$. For example, Korobeinikov studied in \cite{korobeinikov2006lyapunov,korobeinikov2007global} the global stability of SIR and SIRS models that incorporate an incidence function $f(S,I)$, which eliminates the monotony and concavity conditions that were common in previous models.

In this paper, we propose a modification of the SIRS model considered by \citeauthor{li2017threshold} in~\cite{li2017threshold} with a more general incidence function $f(S,I)$. We will study the model given by the following system of differential equations:

\begin{equation}\label{mod2}
\begin{aligned}
	\ds & =  \Lambda - \mu S -f(S,I) +\gamma_1I +\delta R, \\
	\di & =  f(S,I) - (\mu+\gamma_1+\gamma_2+\alpha)I,     \\
	\dr & =  \gamma_2I -(\mu+\delta)R,
\end{aligned}
\end{equation}
where the parameters are:
\begin{itemize}
    \item $\Lambda$: recruitment rate of susceptible individuals.
    \item $\mu$: natural death rate.
    \item $\gamma_1$: transfer rate from the infected class to the susceptible class.
    \item $\gamma_2$: transfer rate from the infected class to the recovered class.
    \item $\alpha$: disease-induced death rate.
    \item $\delta$: immunity loss rate.
\end{itemize}

$\Lambda$ and $\mu$ are assumed to be positive, while the parameters $\alpha$, $\gamma_1$, $\gamma_2$ and $\delta$ are non-negative.

We make the following hypotheses on $f$:

\begin{description}
    \item[(H1)] $f(S,I)=If_1(S,I)$ with $f,f_1\in C^2\left(\mathbb{R}^2_+\to\mathbb{R}_+\right)$ and $f(0,I)=f(S,0)=0$ for all $S,I\ge0$.
    \item[(H2)] $\frac{\partial f_1}{\partial S}(S,I)>0$ and $\frac{\partial f_1}{\partial I}(S,I)\le0$ for all $S,I\ge0$.
    \item[(H3)] $\lim\limits_{I\to 0^+}\frac{f(S,I)}{I}$ exists and is positive for all $S>0$.
\end{description}

It is easy to see that every solution of \eqref{mod2} with non-negative initial conditions remains non-negative. If we add up the three equations of the system, we have
\begin{equation}
\der{(S+I+R)}=\Lambda-\mu(S+I+R)-\alpha I\le\Lambda-\mu(S+I+R),
\end{equation}
which implies that the set
\[\Omega=
\left\{(S,I,R)\in\mathbb{R}^3_+\mid S+I+R\le\frac{\Lambda}{\mu}\right\}
\]
is a positively invariant and attractive set for \eqref{mod2}, so we can restrict our attention to solutions with initial conditions in the feasible region $\Omega$.

For all values of the parameters, System \eqref{mod2} has a disease free equilibrium, which is given by $E_0=(S_0,\ 0,\ 0)$ with $S_0=\Lambda/\mu$.

We will prove now a result which will be needed in the following sections.

\begin{lem}\label{lem1}
    Let
    \begin{equation*}
    \beta=\frac{\mu}{\Lambda}\frac{\partial f(S_0,\ 0)}{\partial I}.
    \end{equation*} Then $\beta>0$ and $f(S,I)\le\frac{\Lambda}{\mu}\beta I$ for all $0\le S,I\le S_0$.
\end{lem}
\begin{proof}
    From \textbf{(H3)}, it follows that
    \[\lim_{I\to 0^+}\frac{f(S,I)}{I}
    =\lim_{I\to 0^+}\frac{f(S,I)-f(S,0)}{I-0}=\frac{\partial f(S,0)}{\partial I}
    \]
    exists and is positive for all $S>0$. In particular, for $S=S_0=\Lambda/\mu$ we have
    \begin{equation*}
    \lim_{I\to 0^+}\frac{f(S_0,I)}{I}
    =\frac{\partial f(S_0,0)}{\partial I},
    \end{equation*}
    so the number $\beta=\frac{\mu}{\Lambda}\frac{\partial f(S_0,\ 0)}{\partial I}$ is well-defined and positive.
    
    Now, we have from \textbf{(H2)} that $f_1(S,I)$ is increasing with respect to $S$. Thus, if we take a fixed $S$ such that $0\le S\le S_0$, then
    \begin{equation}\label{ineq1}
    \frac{f(S,I)}{I}=f_1(S,I)\le f_1(S_0,I)\quad\text{for all }I>0.
    \end{equation}
    On the other hand, from \textbf{(H2)} we also have that $f_1(S,I)$ is non-increasing as a function of $I$. Hence, by \textbf{(H3)} we have
    \begin{equation}\label{ineq2}
    f_1(S_0,I)
    \le f_1(S_0,0)=\lim_{I\to 0^+}f_1(S_0,I)=\lim_{I\to 0^+}\frac{f(S_0,I)}{I}
    =\frac{\partial f(S_0,0)}{\partial I}
    =\frac{\Lambda}{\mu}\beta
    \end{equation}
    for all $I>0$. Then, from \eqref{ineq1} and \eqref{ineq2} we can conclude that
    \begin{equation}
    \frac{f(S,I)}{I}\le\frac{\Lambda}{\mu}\beta
    \quad\text{for $0\le S\le\frac{\Lambda}{\mu}$ and $I>0$},
    \end{equation}
    and thus, $f(S,I)\le\frac{\Lambda}{\mu}\beta I$ for $0\le S,I\le S_0$.
\end{proof}

\section{Basic reproduction number}

In this section, we calculate the basic reproduction number $R_0$ of the model using the next-generation matrix method described in \cite{van2002reproduction}. For that, we write the system as
\[
\frac{\text{d}x}{\text{d}t}=\mathscr{F}(x)-\mathscr{V}(x),
\]
where $x=[I,S,R]^T$,
\begin{equation}\label{fv}
\mathscr{F}(x)=
\begin{bmatrix}
	f(S,I) \\
	0      \\
	0
\end{bmatrix},\quad
\mathscr{V}(x)=
\begin{bmatrix}
	(\mu+\gamma_1+\gamma_2+\alpha)I             \\
	-\Lambda +\mu S +f(S,I) -\gamma_1I-\delta R \\
	-\gamma_2I +(\mu+\delta)R
\end{bmatrix}.
\end{equation}
The disease-free equilibrium is given by $x_0=\left[0,\ \Lambda/\mu,\ 0\right]$, with the variables ordered as $(I,S,R)$. Computing $F$ and $V$ defined by
\[F=\frac{\partial\mathscr{F}_1}{\partial I}(x_0),\qquad
V=\frac{\partial\mathscr{V}_1}{\partial I}(x_0),
\]
we obtain
\begin{equation*}
F=\frac{\partial f(\Lambda/\mu,\ 0)}{\partial I}
=\frac{\Lambda}{\mu}\beta,\quad
V=\mu+\gamma_1+\gamma_2+\alpha,
\end{equation*}
so the next-generation matrix is given by
\begin{equation*}
FV^{-1}=\frac{\Lambda\beta}{\mu(\mu+\gamma_1+\gamma_2+\alpha)}.
\end{equation*}
The basic reproduction number $R_0$ is given by the spectral radius of the matrix $FV^{-1}$. Hence,
\begin{equation}\label{r0}
R_0=\frac{\Lambda\beta}{\mu(\mu+\gamma_1+\gamma_2+\alpha)}.
\end{equation}

According to the general result proved in \cite{van2002reproduction}, we can conclude that the disease-free equilibrium $E_0$ is locally asymptotically stable if $R_0<1$ and it is unstable if $R_0>1$.

\section{Existence of equilibria}

We have already established the existence of the disease-free equilibrium $E_0=(\Lambda/\mu,\ 0,\ 0)$ for System \eqref{mod2} for all values of the parameters. We will now prove the existence of at least one endemic equilibrium in the case when $R_0>1$.

\begin{teo}
    If $R_0>1$, then System \eqref{mod2} has an endemic equilibrium.
\end{teo}

\begin{proof}
    At an equilibrium point $(S,I,R)\in\Omega$ of \eqref{mod2}, the equalities
    \[\Lambda - \mu S -f(S,I) +\gamma_1I +\delta R=0,\quad
    f(S,I) - (\mu+\gamma_1+\gamma_2+\alpha)I=0,\quad
    \gamma_2I -(\mu+\delta)R=0\]
    must hold. The last one implies that $R=\gamma_2I/(\mu+\delta)$. By substituting this expression for $R$ and $f(S,I)=(\mu+\gamma_1+\gamma_2+\alpha)I$ in the first equality, we have
    \begin{equation}
    \Lambda-\mu S-(\mu+\gamma_2+\alpha)I+\frac{\delta\gamma_2I}{\mu+\delta}=0.
    \end{equation}
    Since $f(S,I)=If_1(S,I)$, the equation $f(S,I) - (\mu+\gamma_1+\gamma_2+\alpha)I=0$ is equivalent to $f_1(S,I)-(\mu+\gamma_1+\gamma_2+\alpha)=0$ for $I>0$. Hence, in order to determine the positive equilibria of \eqref{mod2}, we must solve the system given by
    \begin{equation}\label{qq}
    S=\frac{\Lambda}{\mu}-\frac{1}{\mu}\left(\mu+\gamma_2+\alpha-\frac{\delta\gamma_2}{\mu+\delta}\right)I,\qquad
    f_1(S,I)=\mu+\gamma_1+\gamma_2+\alpha.
    \end{equation}
    These two equalities define a straight line $q_1$ and a curve $q_2$, respectively, on the $IS$ plane. Notice that $(\mu+\gamma_2+\alpha)(\mu+\delta)>\delta\gamma_2$ and thus, $\mu+\gamma_2+\alpha-\frac{\delta\gamma_2}{\mu+\delta}>0$, so $q_1$ has negative slope. From the hypothesis \textbf{(H2)} we have that $\frac{\partial f_1}{\partial S}(S,I)>0$, so $f_1(S,I)$ is increasing with respect to $S$. Thus, by the implicit function theorem, the second equality in \eqref{qq} defines a function $S=h(I)$, which is positive for $I>0$. Let $S_0=\Lambda/\mu$, $S_*=h(0)$ and $I_0=\Lambda/\left(\mu+\gamma_2+\alpha-\frac{\delta\gamma_2}{\mu+\delta}\right)$. Then $q_1$ and $q_2$ intersect the $S$ axis at the points $(0,S_0)$ and $(0,S_*)$, respectively, and $q_1$ intersects the $I$ axis at $(I_0,0)$ (see Figure \ref{fig:qq}). The function $h(I)$ either exists and is continuous for $I\in\left(0,\ I_0\right)$ or reaches infinity in this interval. In either case, it is clear that $q_1$ and $q_2$ have at least one positive point of intersection if $S_*<S_0$. Since $f(S,I)$ grows monotonically with respect to $S$, then $S_0/S_*>1$ holds whenever
    \begin{align*}
    	\lim_{I\to 0^+}\frac{f(S_0,I)}{f(S_*,I)} & =\lim_{I\to 0^+}\frac{f(S_0,I)}{If_1(S_*,I)}=\lim_{I\to 0^+}\frac{f(S_0,I)}{(\mu+\gamma_1+\gamma_2+\alpha)I}                                       \\
    	                                         & =\frac{1}{\mu+\gamma_1+\gamma_2+\alpha}\cdot\lim_{I\to 0^+}\frac{f(S_0,I)}{I} =\frac{1}{\mu+\gamma_1+\gamma_2+\alpha}\cdot\frac{\Lambda}{\mu}\beta \\
    	                                         & =R_0>1.
    \end{align*}
\end{proof}

\begin{figure}[h]\centering
    \begin{tikzpicture}
        \draw[->] (0,0) node[anchor=north east] {$O$} -- (5,0) node[anchor=north] {$I$};
        \draw[->] (0,0) -- (0,5) node[anchor=east] {$S$};
        \draw[thick] (0,1) to[out=0,in=240] (4.5,4)
        node[anchor=south] {$q_2$};
        \draw[thick] (-.5,4.5) -- (4,-.5)
        node[anchor=north] {$q_1$};
        \draw (0.1,3.9) node[anchor=east] {$S_0$};
        \draw (0.1,1) node[anchor=east] {$S_*$};
        \draw (3.8,0) node[anchor=north east] {$I_0$};
    \end{tikzpicture}
    \caption{Graphs of the line $q_1$ and the curve $q_2$.}
    \label{fig:qq}
\end{figure}
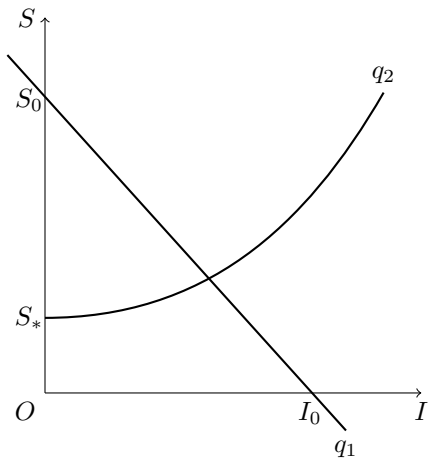

\section{Global stability}

In this section, we will show the global stability of the DFE for System \eqref{mod2} using LaSalle's invariance principle, and we will establish   the global stability of the endemic equilibrium by means of Lyapunov's direct method.

\begin{teo}\label{teo:globDFE}
    If $R_0\le1$, then the disease-free equilibrium $E_0$ of System \eqref{mod2} is globally asymptotically stable in $\Omega$ and there is no endemic equilibrium.
\end{teo}

\begin{proof}
    We will use the function $V(S,I,R)=I$. The derivative of $V$ along the solutions of \eqref{mod2} is given by
    \[\der{V}=
    f(S,I) - (\mu+\gamma_1+\gamma_2+\alpha)I.
    \]
    
    From Lemma \ref{lem1}, we have that $f(S,I)\le \frac{\Lambda}{\mu}\beta I$. Then
    \begin{align*}
    	\der{V} & \le \left[\frac{\Lambda}{\mu}\beta - (\mu+\gamma_1+\gamma_2+\alpha)\right]I                                                \\
    	        & =(\mu+\gamma_1+\gamma_2+\alpha)\frac{\Lambda\beta-\mu(\mu+\gamma_1+\gamma_2+\alpha)}{\mu(\mu+\gamma_1+\gamma_2+\alpha)}I \\
    	        & =(\mu+\gamma_1+\gamma_2+\alpha)(R_0-1)I.
    \end{align*}
    If $R_0\le1$, then we have $\der{V}\le(\mu+\gamma_1+\gamma_2+\alpha)(R_0-1)I\le0$. Suppose that $(S(t),I(t),R(t))$ is a solution of \eqref{mod2} contained entirely in the set $M=\left\{(S,I,R)\in\Omega\mid \der{V}=0\right\}$. Then $\di=0$ and, from the above inequalities, we have that $I=0$. Thus, the largest positively invariant set contained in $M$ is the plane $I=0$. By LaSalle's invariance principle, this implies that all solutions in $\Omega$ approach the plane $I=0$ as $t\to\infty$. On the other hand, solutions of $\eqref{mod2}$ contained in such plane satisfy $\ds=\Lambda-\mu S+\delta R$, $\dr=-(\mu+\delta)R$, which implies that $S\to\frac{\Lambda}{\mu}$ and $R\to0$ as $t\to\infty$, that is, all of these solutions approach $E_0$. Therefore, we conclude that $E_0$ is globally attractive in $\Omega$. Also, since the derivative of a Lyapunov function must be zero at all equilibrium points and $E_0$ is the only equilibrium in the plane $I=0$, this implies that \eqref{mod2} has no equilibria in $\Omega$ different from $E_0$.\bigskip
\end{proof}

Let $E_1=(S^*,\ I^*,\ R^*)$ be an endemic equilibrium of \eqref{mod2}. We will now establish the global stability of $E_1$ by constructing a Lyapunov function and using Lyapunov's direct method.

In \cite{tang2017new}, the authors considered an SIRS model with a general incidence rate of the form $\beta f(S)g(I)$ and proposed a Lyapunov function for the stability of the endemic equilibrium. We will modify that function to adapt it to our model. For that, we need to introduce the following two assumptions, which are similar to those used in \cite{tang2017new}:
\begin{description}
    \item[(A1)] $(2\mu+\alpha)(\mu+\delta)>\mu\gamma_2$.
    \item[(A2)] Let $(S^*,\ I^*,\ R^*)$ be an endemic equilibrium of \eqref{mod2} and define the function
    \[G(u,v)=\frac{1}{u-S^*}\left[\frac{f(u,v)}{v}-\frac{f(S^*,I^*)}{I^*}\right].
    \]
    Then there exists a positive constant $k_1$ such that
    \begin{equation}\label{con1}
    2\mu(\mu+\alpha)>
    \left[2\mu+\alpha - k_1G(u,v)\right]^2
    \end{equation}
    for all $u\ne S^*$ and $0\le u,v\le\frac{\Lambda}{\mu}$.
\end{description}

\begin{teo}\label{teo:globEND}
    Suppose that assumptions \textbf{(A1)} and \textbf{(A2)} hold. If $R_0>1$, then the endemic equilibrium $E_1$ of System \eqref{mod2} is unique and globally asymptotically stable in $\Omega$.
\end{teo}

\begin{proof}
    Consider the function
    \begin{equation*}
    	V(S,I,R) = \frac12(S-S^*+I-I^*+R-R^*)^2 + k_1\left(I-I^*-I^*\ln\frac{I}{I^*}\right) + \frac{k_2}{2}(R-R^*)^2,
    \end{equation*}
    where $k_1$ is the positive constant from assumption \textbf{(A2)} and $k_2>0$ will be determined later. It is clear that $V(S^*,I^*,R^*)=0$ and that $V(S,I,R)>0$ for all $(S,I,R)\in\Omega-\{E_1\}$. Therefore, $V(S,I,R)$ is a Lyapunov function. Now, we can rewrite \eqref{mod2} as
    \begin{equation}\label{mod2b}
    \begin{aligned}
    	\ds & =  -\mu(S-S^*) -\big[f(S,I)-f(S^*,I^*)\big] +\gamma_1(I-I^*) +\delta (R-R^*), \\
    	\di & =  \big[f(S,I)-f(S^*,I^*)\big] - (\mu+\gamma_1+\gamma_2+\alpha)(I-I^*),       \\
    	\dr & =  \gamma_2(I-I^*) -(\mu+\delta)(R-R^*).
    \end{aligned}
    \end{equation}
    Using also that $(\mu+\gamma_1+\gamma_2+\alpha)=f(S^*,I^*)/I^*$ and
    \[\der{I}=f(S,I)-(\mu+\gamma_1+\gamma_2+\alpha)I=f(S,I)-\frac{I}{I^*}f(S^*,I^*),
    \]
    we can calculate the time derivative of $V(S,I,R)$ along the solutions of \eqref{mod2b}:
    {\footnotesize 
    \begin{align*}
    	\der{V} & = (S-S^*+I-I^*+R-R^*)\der{(S+I+R)} +  k_1\left(1-\frac{I^*}{I}\right)\der{I} + k_2(R-R^*)\der{R}                              \\
    	        & = (S-S^*+I-I^*+R-R^*)\big[-\mu(S-S^*)-(\mu+\alpha)(I-I^*)-\mu(R-R^*)\big]                                                     \\
    	        & \quad + k_1\frac{I-I^*}{I}\Big[f(S,I)-\frac{I}{I^*}f(S^*,I^*)\Big] + k_2(R-R^*)\big[\gamma_2(I-I^*) -(\mu+\delta)(R-R^*)\big] \\
    	        & = - \mu(S-S^*)^2-(\mu+\alpha)(I-I^*)^2 -\big[\mu+k_2(\mu+\delta)\big](R-R^*)^2-(2\mu+\alpha)(S-S^*)(I-I^*)                    \\
    	        & \quad -2\mu(S-S^*)(R-R^*) -(2\mu+\alpha)(I-I^*)(R-R^*)+\gamma_2k_2(I-I^*)(R-R^*)                                              \\
    	        & \quad + k_1(I-I^*)\left[\frac{f(S,I)}{I}-\frac{f(S^*,I^*)}{I^*}\right],
    \end{align*}}
    We can take $k_2=(2\mu+\alpha)/\gamma_2$, so that the terms in $(I-I^*)(R-R^*)$ cancel out. Thus,
    {\footnotesize 
    \begin{equation*}
    \begin{aligned}
    	\der{V} & = - \mu(S-S^*)^2-(\mu+\alpha)(I-I^*)^2 -\big[\mu+k_2(\mu+\delta)\big](R-R^*)^2-(2\mu+\alpha)(S-S^*)(I-I^*)                                                         &  \\
    	        & \quad -2\mu(S-S^*)(R-R^*) + k_1(I-I^*)\left[\frac{f(S,I)}{I}-\frac{f(S^*,I^*)}{I^*}\right]                                                                         &  \\
    	        & = -\Bigg\{\frac{\mu}{2}(S-S^*)^2 + \left\{2\mu+\alpha - \frac{k_1}{S-S^*}\left[\frac{f(S,I)}{I}-\frac{f(S^*,I^*)}{I^*}\right]\right\}(S-S^*)(I-I^*)             \\ & \quad +(\mu+\alpha)(I-I^*)^2\Bigg\} -\Bigg\{\frac{\mu}{2}(S-S^*)^2 + 2\mu(S-S^*)(R-R^*) +\big[\mu+k_2(\mu+\delta)\big](R-R^*)^2
    \Bigg\} \\
    	        & = - \left[S-S^*,\ I-I^*\right] Q \left[S-S^*,\ I-I^*\right]^T - \left[S-S^*,\ R-R^*\right] P \left[S-S^*,\ R-R^*\right]^T,                                         &
    \end{aligned}
    \end{equation*}}
    where
    \[P=
    \begin{bmatrix}
    	\frac{\mu}{2} & \mu                 \\
    	\mu           & \mu+k_2(\mu+\delta)
    \end{bmatrix},
    \]
    \[Q=
    \begin{bmatrix}
    	\frac{\mu}{2}                          & \frac12\big[2\mu+\alpha-k_1G(S,I)\big] \\
    	\frac12\big[2\mu+\alpha-k_1G(S,I)\big] & \mu+\alpha
    \end{bmatrix}.\]
    Assumptions \textbf{(A1)} and \textbf{(A2)} imply that all leading principal minors of $P$ and $Q$ are positive. Therefore, $P$ and $Q$ are positive-definite matrices.
    
    Let $w=\left[S-S^*,\ I-I^*\right]^T$ and $v=\left[S-S^*,\ R-R^*\right]^T$. Then $w^TQw\ge0$ and $v^TPv\ge0$, so
    \[\der{V}\le-w^TQw-v^TPv\le0.\]
    Suppose now that $\der{V}(S,I,R)=0$ for some $(S,I,R)\in\Omega$. Then $w^TQw=0$ and $v^TPv=0$, and the positive-definiteness of $P$ and $Q$ implies that $w=0$ and $v=0$. Thus, $\left[S-S^*,I-I^*\right]^T=0$ and $\left[S-S^*,R-R^*\right]^T=0$, so $(S,I,R)=(S^*,I^*,R^*)=E_1$.
    
    This implies that the set $\{(S,I,R)\in\Omega\mid\der{V}(S,I,R)=0\}$ consists only of the point $E_1$. Hence, by Lyapunov's direct method, we conclude that $E_1$ is globally asymptotically stable, and this also proves the uniqueness of the endemic equilibrium.
\end{proof}

\begin{rem}
    The inequality in \textbf{(A1)} can be rewritten as $2\mu^2+(\alpha+2\delta-\gamma_2)\mu+\alpha\delta>0$. Since all parameters are assumed to be positive, we can see that \textbf{(A1)} holds if $\mu$ or $\gamma_2$ are close to zero, that is, if the natural death rate or the transfer rate from infected to recovered classes are small compared to the other parameters.
\end{rem}

\begin{rem}
    Theorems \ref{teo:globDFE} and \ref{teo:globEND} show that, under assumptions \textbf{(A1)} and \textbf{(A2)}, the prevalence or extinction of the disease is completely determined by the basic reproduction number, since the DFE is globally stable for $R_0\le1$, while the unique endemic equilibrium is globally stable for $R_0>1$.
\end{rem}

\section{Numerical simulations}

In this section we present some numerical simulations of the solutions for System \eqref{mod2}. We will consider a particular case for the incidence function $f(S,I)$ in order to check out the theoretical results proved above.

\begin{ejm}\label{ex1}
    In \cite{sun2007global}, the authors consider the incidence rate described by the non-linear function $f(S,I)=kI^pS^q$, $p>0$, $q>0$, which includes the traditional bilinear incidence from the principle of mass action ($p=q=1$). It can be verified that the assumptions \textbf{(H1)} and \textbf{(H2)} hold for $q\ge1$ and $p=1$. In such case, we have $\lim\limits_{I\to 0^+}\frac{f(S,I)}{I}=kS^q$, so \textbf{(H3)} also holds. Thus, for the incidence function
    $$f(S,I)=kIS^q,$$
    the basic reproduction number of System \eqref{mod2} is given by
    \[
    R_0=\frac{\Lambda^qk}{\mu^q(\mu+\gamma_1+\gamma_2+\alpha)}.
    \]
    
    Figure \ref{fig:a} shows several solutions of \eqref{mod2} with different initial conditions when $R_0<1$, and it can be seen that the DFE is globally asymptotically stable.
    
    \begin{figure}
        \begin{subfigure}[b]{0.75\textwidth}
            \includegraphics[width=\textwidth]{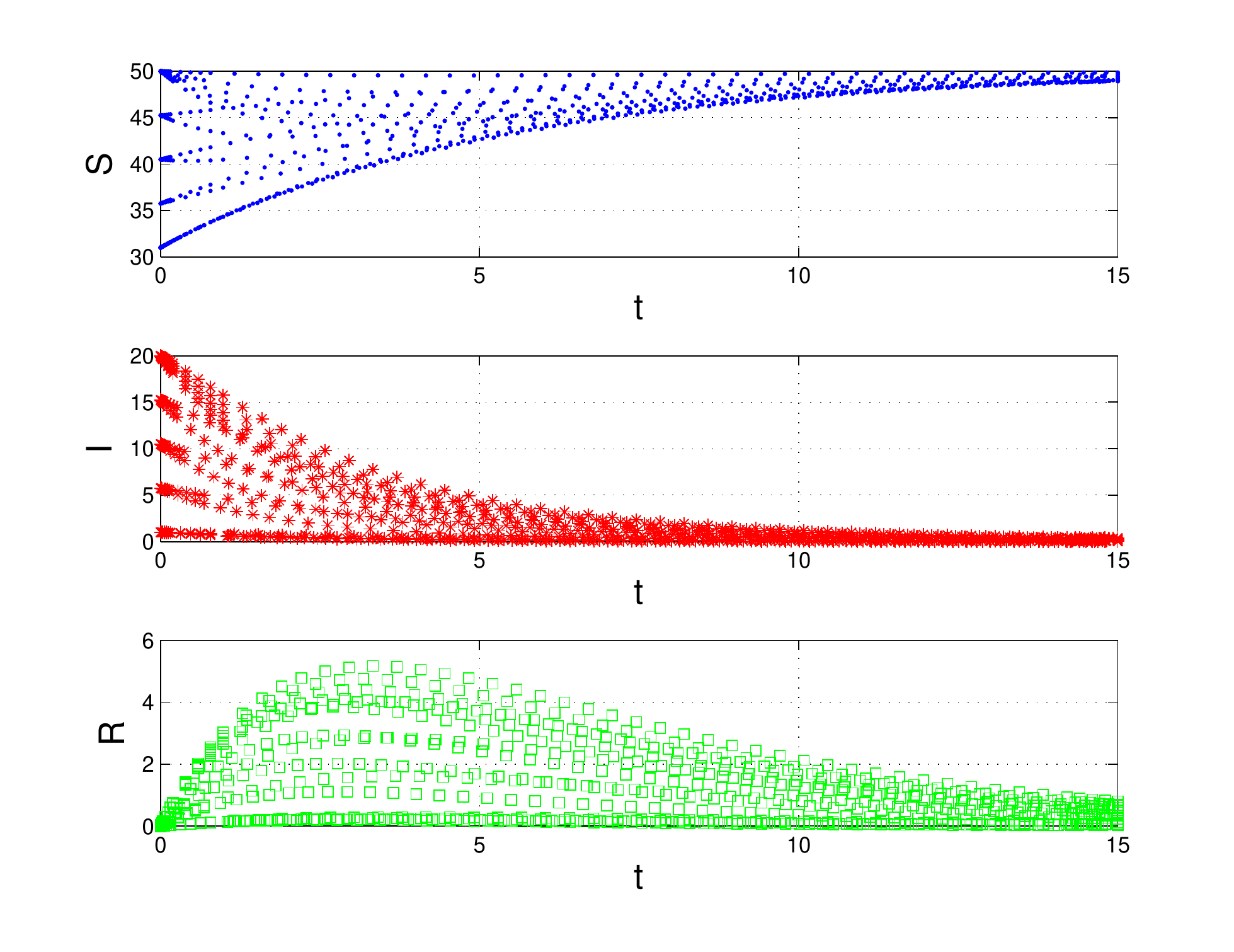}
            \caption{}
        \end{subfigure}
        \begin{subfigure}[b]{0.75\textwidth}
            \includegraphics[width=\textwidth]{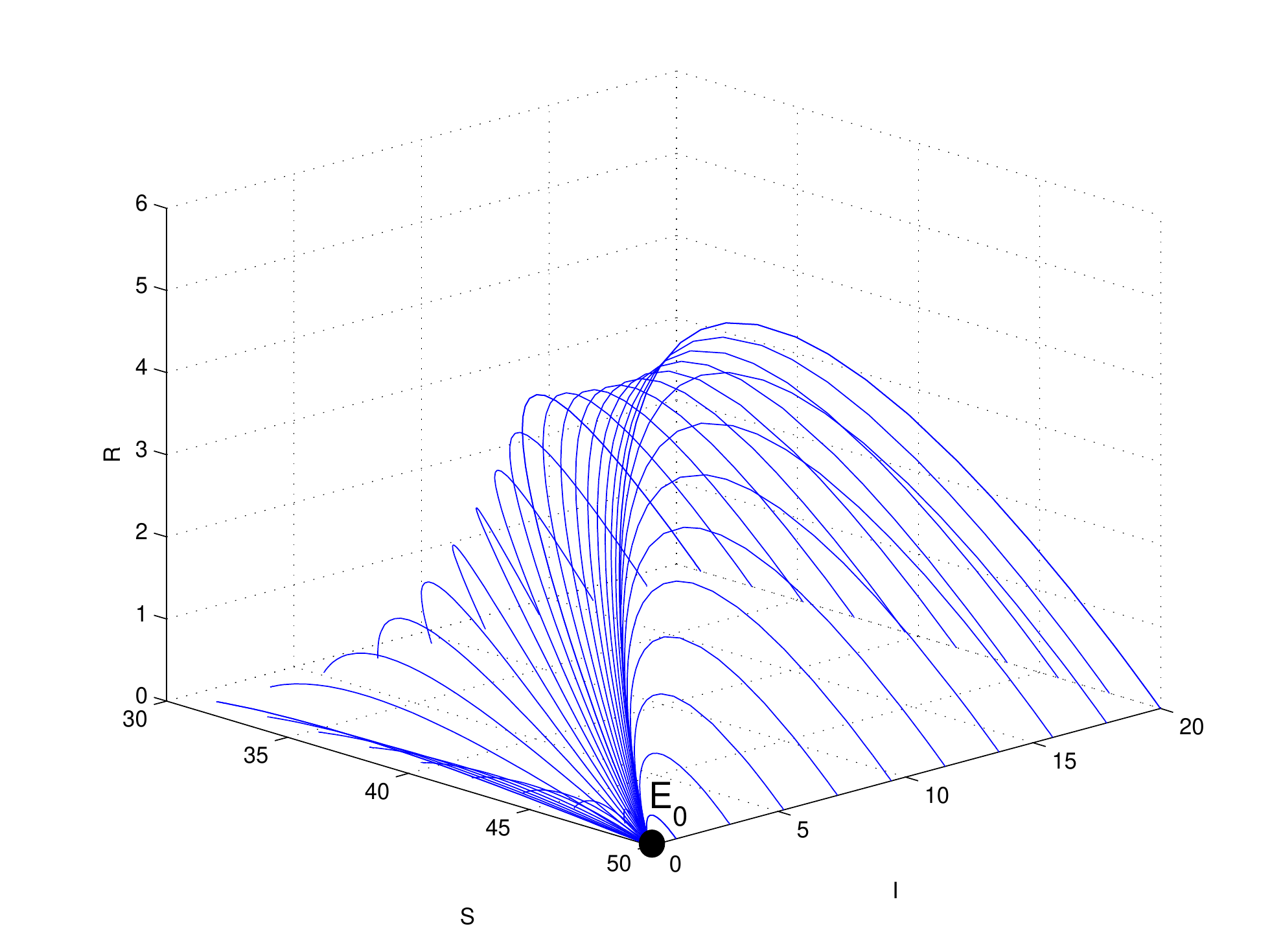}
            \caption{}
        \end{subfigure}
        \caption{Time-series (\textsc{a}) and 3-dimensional phase portrait (\textsc{b}) of System \eqref{mod2} for $f(S,I)=kIS^q$, using the parameters $k=0.0002$, $q=2$, $\Lambda=10$, $\mu=0.2$, $\gamma_1=0.2$, $\gamma_2=0.2$, $\alpha=0.1$, $\delta=0.1$ such that $R_0=0.7143<1$. The disease-free equilibrium is $E_0=(50,\ 0,\ 0)$. For this set of parameters, the system has no endemic equilibria and the DFE is globally asymptotically stable.}\label{fig:a}
    \end{figure}
    
    Taking a different set of parameters such that $R_0>1$, we obtain the solutions shown in Figure \ref{fig:b}. With these parameters, we have $(\mu+\alpha)(\mu+\delta)=0.15>\mu\gamma_2=0.04$, so \textbf{(A1)} holds. We will verify now the condition \textbf{(A2)}: with the coordinates of the endemic equilibrium $(S^*,\ I^*, R^*)=(29.5804,\ 9.4244,\ 6.2830)$, we have
    $$G(u,v)=\frac{0.0008u^2-0.7}{u-29.5804}.
    $$
    If we choose the positive constant $k_1=7$, then Inequality \eqref{con1} becomes
    $$0.12>\left[0.5-\frac{7\left(0.0008u^2-0.7\right)}{u-29.5804}\right]^2=:h(u).
    $$
    Figure \ref{fig:uhu} shows the graph of $h(u)$, where it can be seen that $h(u)<0.12$ for $0\le u\le50$. Therefore, from Theorem \ref{teo:globEND} we conclude that the endemic equilibrium is globally asymptotically stable.
    
    \begin{figure}
        \begin{subfigure}[b]{0.75\textwidth}
            \includegraphics[width=\textwidth]{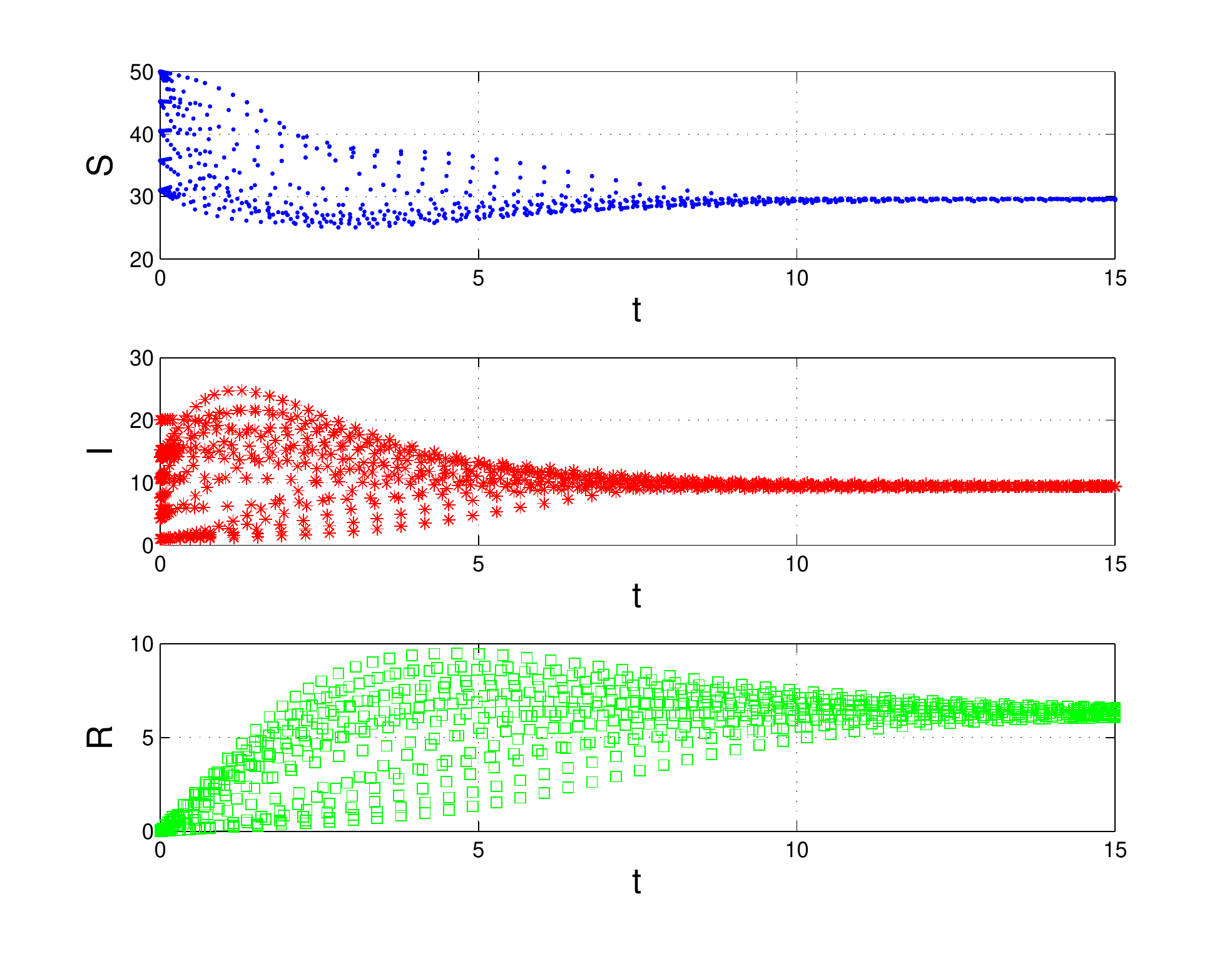}
            \caption{}
        \end{subfigure}
        \begin{subfigure}[b]{0.75\textwidth}
            \includegraphics[width=\textwidth]{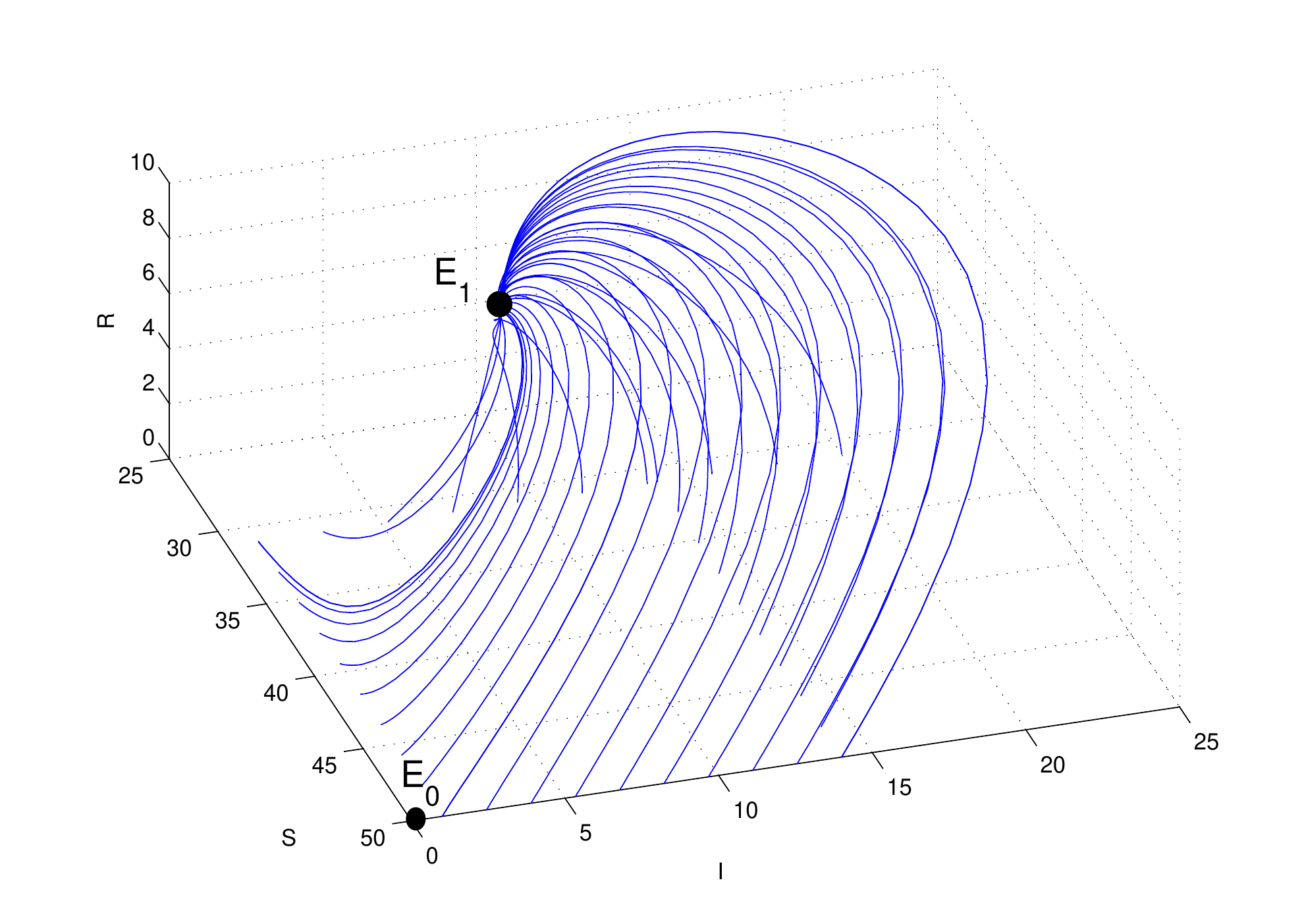}
            \caption{}
        \end{subfigure}
        \caption{Time-series (\textsc{a}) and 3-dimensional phase portrait (\textsc{b}) of System \eqref{mod2} for $f(S,I)=kIS^q$, using the parameters $k=0.0008$, $q=2$, $\Lambda=10$, $\mu=0.2$, $\gamma_1=0.2$, $\gamma_2=0.2$, $\alpha=0.1$, $\delta=0.1$ such that $R_0=2.8571>1$. The disease-free equilibrium is $E_0=(50,\ 0,\ 0)$, which is unstable. The unique endemic equilibrium is $E_1=(29.5804,\ 9.4244,\ 6.2830)$ and is globally asymptotically stable.}\label{fig:b}
    \end{figure}
    
    \begin{figure}
        \centering
        \includegraphics[width=0.7\linewidth]{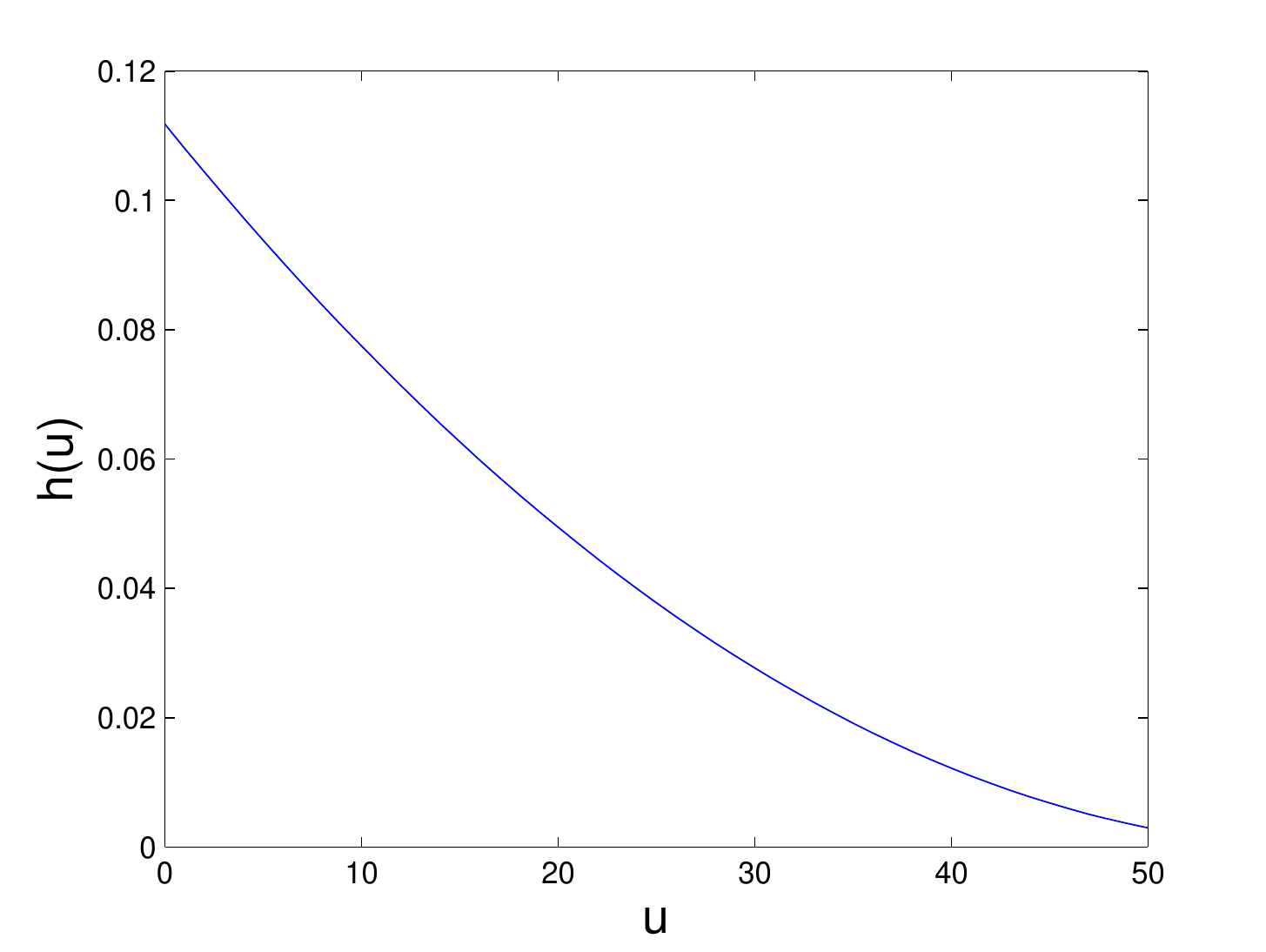}
        \caption{Graph of $u$ and $h(u)$ as defined in example \ref{ex1} with the parameters from figure \ref{fig:b}.}
        \label{fig:uhu}
    \end{figure}

\end{ejm}

\section{Discussion}

In this paper we considered an SIRS epidemic model with a non-linear incidence rate of the form $f(S,I)$. We have proved that if the basic reproduction number is greater than 1 and some conditions for the parameters of the system hold, then the endemic equilibrium is unique and globally asymptotically stable. We extended the work done in \cite{li2017threshold} by considering a more general incidence rate than the function $Sf(I)$ studied there, at the expense of introducing the assumptions \textbf{(A1)} and \textbf{(A2)} in order to establish the global stability of the endemic equilibrium. The first of these assumptions is an inequality of the parameters which is verified immediately whenever the natural death rate or the transfer rate from infected to recovered are close to zero. This implies that, under such conditions, the basic reproduction number acts as a threshold parameter which completely determines the prevalence or extinction of the disease. An interesting question which could be studied in the future would be to investigate the effects that introducing a treatment function in system \eqref{mod2} has on the global dynamics of the disease.

\printbibliography

\end{document}